\documentclass{amsart}

\usepackage{ae}
\usepackage[all]{xy}
\usepackage{graphicx,amsfonts,amssymb,amsmath}
\usepackage{eurofont}
\usepackage{natbib}
\usepackage{enumerate}

\usepackage[latin1]{inputenc}
\usepackage[T1]{fontenc}
\usepackage[francais]{babel}
\usepackage[cyr]{aeguill}

\newcommand{\chapeau}{{\rlap{\smash{\hbox{\lower4pt\hbox{\hskip1pt$\widehat{\phantom{u}}$}}}}}\mbox{ }}
\usepackage[OT2,T1]{fontenc}
\DeclareSymbolFont{cyrletters}{OT2}{wncyr}{m}{n}
\DeclareMathSymbol{\sha}{\mathalpha}{cyrletters}{"58}
\input cyracc.def

 \newtheorem{thm}{Théorème}
 \newtheorem*{ques}{\textit{Question}}

 \newtheorem{lem}[thm]{Lemme}
 
 \theoremstyle{definition}
 
 \theoremstyle{definition}
 
 \theoremstyle{remark}
 
 \theoremstyle{remark}
 \newtheorem{rem}[thm]{Remarque}
 \theoremstyle{remark}

 \numberwithin{equation}{subsection}

 \newcommand{\To}{\longrightarrow}

 \renewcommand{\P}{\mathbb{P}}
 \newcommand{\Q}{\mathbb{Q}}
 \newcommand{\Z}{\mathbb{Z}}
 
 \newcommand{\K}{\mathcal{K}}

 \newcommand{\coker}{\textup{Coker}}

 \newcommand{\Br}{\textup{Br}}
 
 \newcommand{\CH}{\textup{CH}}
 \renewcommand{\H}{\textup{H}}
 \newcommand{\Hom}{\textup{Hom}}
 \newcommand{\Spec}{\textup{Spec}}
 
 \newcommand{\inv}{\textup{inv}}
 \newcommand{\E}{\textup{E}}
 
 \newcommand{\BM}{Brauer\textendash Manin }
 \renewcommand{\Spec}{\textup{Spec}}

\usepackage{hyperref}

\begin{document}

\title[]
{Approximation faible pour les 0-cycles sur un produit de variétés rationnellement connexes}

\author{ Yongqi LIANG  }

\address{Yongqi LIANG
\newline B\^atiment Sophie Germain,
\newline Universit\'e Paris Diderot - Paris 7,
\newline Institut de Math\'ematiques de Jussieu -  Paris Rive Gauche,
 \newline  75013 Paris,\newline
 France}

\email{yongqi.liang@imj-prg.fr}

\thanks{\textit{Mots clés} : zéro-cycles, approximation faible, obstruction de \BM}


\date{\today.}



\maketitle

\begin{abstract}
Considérons l'approximation faible de 0-cycles sur une variété propre lisse définie  sur un corps de nombres, elle est conjecturée d'être contrôlée par son groupe de Brauer.  Soit $X$ une surface de Châtelet ou une compactification lisse d'un espace homogène d'un groupe algébrique linéaire connexe à stabilisateur connexe. Soit $Y$ une variété rationnellement connexe. Nous montrons que 
l'approximation faible de 0-cycles sur le produit $X\times Y$ est contrôlée par son groupe de Brauer si c'est le cas pour  $Y$ après toute extension finie du corps de base. Nous ne supposons l'existence de 0-cycles  de degré $1$ ni sur $X$ ni sur $Y$.
\end{abstract}

\small
\footnotesize
\begin{center} \textbf{Weak approximation for 0-cycles on a product of rationally connected varieties}
\end{center}

\noindent\textit{Abstract.} Consider weak approximation for 0-cycles on a smooth proper variety defined over a number field, it is conjectured to be controlled by its Brauer group. Let $X$ be a Châtelet surface or a smooth compactification of a homogeneous space of a connected linear algebraic group with connected stabilizer. Let $Y$ be a rationally connected variety. We prove that weak approximation for 0-cycles on the product $X\times Y$ is controlled by its Brauer group if it is the case for  $Y$ after every finite extension of the base field. We do not suppose the existence of 0-cycles of degree $1$ neither on $X$ nor on $Y$.

\normalsize


\section{Introduction}
Soit $k$ un corps de nombres. On considère l'approximation faible pour les 0-cycles sur les $k$-variétés propres lisses et géométriquement connexes $V$. Il est conjecturé que l'obstruction de Brauer\textendash Manin est la seule obstruction à l'approximation faible pour les 0-cycles sur toutes telles variétés, \cite{CTSansuc81,KatoSaito86,CT95}. Grosso modo, on espère que la suite
$$\varprojlim_n\CH_0(V)/n\to\prod_{v\in\Omega_{k}}\varprojlim_n\CH'_0(V_{k_{v}})/n\to \Hom(\Br(V),\mathbb{Q}/\mathbb{Z})\leqno (\E)$$
soit exacte pour toute variété propre lisse, voir \S\ref{notation} pour les notations et voir \cite{Wittenberg} pour plus d'informations.

Dans ce texte, nous nous restreignons au cas d'un produit de variétés $V=X\times Y$. Dans l'article profond de Skorobogatov et Zarhin \cite{SkZarhin}, ils ont démontré une relation entre les ensembles de Brauer\textendash Manin 
$$\left[\prod_{v}X\times Y(k_{v})\right]^{\Br}=\left[\prod_{v}X(k_{v})\right]^{\Br}\times\left[\prod_{v}Y(k_{v})\right]^{\Br}.$$ Par conséquent, si l'obstruction de Brauer\textendash Manin est la seule  à l'approximation faible pour les points rationnels sur $X$ et sur $Y$, alors il en va de même sur $X\times Y$. Avec très peu de modifications, l'argument de Skorobogatov\textendash Zarhin montre aussi la surjectivité de l'application naturelle 
$$\left[\prod_{v\in\Omega_{k}}\CH_{0}^{1}(X_{k_{v}}\times Y_{k_{v}})\right]^{\Br}\to\left[\prod_{v\in\Omega_{k}}\CH_{0}^{1}(X_{k_{v}})\right]^{\Br}\times\left[\prod_{v\in\Omega_{k}}\CH_{0}^{1}(Y_{k_{v}})\right]^{\Br},$$
où $[\prod\CH_{0}^{1}(V_{k_{v}})]^{\Br}$ désigne l'ensemble des familles de classes de 0-cycles locaux de degré $1$ qui sont orthogonales au groupe de Brauer de $V$. Par contre, dans le cadre de 0-cycles de degré $1$, ce n'est pas clair si cette application est une bijection.
Nous sommes intéressés par la question suivante concernant l'approximation faible de 0-cycles de degré quelconque.
\begin{ques}\label{mainquestion}
Soient $X$ et $Y$ des variétés propres lisses. Supposons que la suite $(\E)$ est exacte pour $X$ et $Y$, est-elle exacte pour $X\times Y$?
\end{ques}

Si $X$ est $k$-rationnelle, il est évident que la exactitude de la suite $(\E)$  pour $Y$ entraîne celle pour $X\times Y$.
Dans leur article récent \cite{HarWit}, comme un cas très particulier du résultat principal, Harpaz et Wittenberg ont obtenu le résultat suivant.

\begin{thm}[Harpaz\textendash Wittenberg]
Soit $X$ une courbe propre lisse sur $k$ telle que le groupe de Tate\textendash Shafarevich de sa jacobienne est fini.
Soit $Y$ une variété rationnellement connexe sur $k$.

Si la suite $(\E)$ est exacte pour $Y$, alors $(\E)$ est aussi exacte pour $X\times Y$.
\end{thm}

Dans ce texte, nous obtenons l'énoncé suivant.
\begin{thm}\label{thmcase}
Soit $X$ une des $k$-variétés suivantes
\begin{itemize}
\item une surface de Châtelet,
\item une compactification lisse d'un espace homogène d'un groupe algébrique linéaire connexe à stabilisateur connexe,
\item une compactification lisse d'un espace homogène d'un groupe algébrique semi-simple simplement connexe à stabilisateur abélien.
\end{itemize}
Soit $Y$ une variété rationnellement connexe sur $k$.

Si la suite $(\E)$ est exacte pour $Y_{K}$ pour toute extension finie $K$ de $k$, alors $(\E)$ est  exacte pour $X\times Y$.
\end{thm}

Lorsque $Y$ admet un 0-cycle de degré $1$, ce théorème est une conséquence immédiate  d'un résultat récent du auteur   \cite[Thm. 2.5]{Liang-fibRC}. L'exactitude de $(\E)$ concerne l'approximation faible de 0-cycles de tout degré, elle est significative même sans l'existence de 0-cycles de degré $1$. Dans ce texte nous ne supposons pas cette existence.

Ce n'est pas clair si on peut remplacer l'hypothèse posée sur toute extension finie par une hypothèse sur le corps de base.

\section{Démonstration du théorème}\label{proof}

\subsection{Notations}\label{notation}

Dans ce texte, le corps de base $k$ est  un corps de nombres 
et les variétés concernées sont  supposées propres lisses et géométriquement connexes sur $k$. Soit $\Omega_{k}$ l'ensemble des places de $k$. Pour un sous-ensemble $S$ de $\Omega_{k}$ et une extension finie $K$ de $k$, la notation $S\otimes_{k}K$ désigne le sous-ensemble des places de $K$ qui se trouvent au-dessus	des places appartenant à $S$. Soit $M$ un groupe abélien, on note par $M_{n}$ et $M/n$ le noyau et le conoyau de la multiplication par $n\in\Z$. 

Soit $V$ une $k$-variété propre lisse, notons par $\CH_{0}(V)$ le groupe de Chow des 0-cycles sur $V$.  
Notons par $\Br (V)=\H^{2}_{\scriptsize{\textup{\'{e}t}}}(V,\mathbb{G}_{\textup{m}})$ le groupe de Brauer cohomologique de $V$. 
On peut définir un accouplement (dit de Brauer\textendash Manin) \cite{Manin,CT95} 
$$\prod_{v\in\Omega_{k}}\CH'_0(V_{k_{v}})\times\Br(V)\to\Q/\Z,$$
où  $\CH'_{0}(-)$ désigne le groupe de Chow modifié. Par définition, $\CH'_0(V_{k_{v}})$ est le groupe de Chow usuel si $v$ est une place non-archimédienne, et sinon $\CH'_0(V_{k_{v}})=\coker[N_{\bar{k}_{v}|k_{v}}:\CH_{0}(V_{\bar{k}_{v}})\to \CH_{0}(V_{k_{v}})]$. On en déduit la suite
$$\varprojlim_n\CH_0(V)/n\to\prod_{v\in\Omega_{k}}\varprojlim_n\CH'_0(V_{k_{v}})/n\to \Hom(\Br(V),\mathbb{Q}/\mathbb{Z}).\leqno (\E)$$
En supposant l'existence d'un 0-cycle de degré $1$, l'exactitude de $(\E)$ implique que l'obstruction de \BM est la seule au principe de Hasse et à l'approximation faible pour les 0-cycles de degré $\delta$ pour tout $\delta\in\Z$, \cite[Rem. 2.2.2]{Liang4}.
On renvoie à \cite{Wittenberg} et \cite{Liang4}  pour la définition de ces suites et les terminologies concernant l'approximation faible pour les 0-cycles.

\subsection{Démonstration du théorème}

Grâce à  \cite[Thm. 8.11]{chateletsurfaces} et \cite[Cor. 2.5]{Borovoi96}, le Théorème \ref{thmcase} résulte du théorème suivant.

\begin{thm}\label{mainthm}
Soient $X,Y$ des variétés géométriquement rationnellement connexes définies sur $k$. Pour une extension finie $L$ de $k$, notons par $\mathcal{K}_{L}$ l'ensemble des extensions finies $K$ de $k$ qui sont linéairement disjointes de $L$ sur $k$.

Pour toute $K\in\mathcal{K}_{L}$, supposons que l'obstruction de Brauer\textendash Manin est la seule obstruction à l'approximation faible pour les points rationnels sur $X_{K}$ et que $(\E)$ est exacte pour $Y_{K}$. 

Alors $(\E)$ est exacte pour $X\times Y$.
\end{thm}

La méthode de sa démonstration est basée sur la preuve du \cite[Thm. 2.5]{Liang-fibRC}, certains détails sont modifiés pour évider l'utilisation de l'existence d'un 0-cycle de degré $1$ sur $Y$: premièrement, dans notre cas la comparaison de groupes de Brauer de fibres est valable sans l'existence d'un 0-cycle de degré $1$ sur la fibre générique; deuxièmement, seulement d'éléments non-ramifiés de groupes de Brauer apparaissent dans la preuve et on n'a plus besoin alors du lemme formel de Harari.

On présente ici quelques lemmes connus qui seront appliqués pendant la preuve du théorème.

\begin{lem}[Lemme de déplacement relatif, {\cite[p. 89]{CT-SD}, \cite[p. 19]{CT-Sk-SD}}]\label{moving lemma}
Soit $\pi:X\to\mathbb{P}^1$ un morphisme dominant défini sur 
$\mathbb{R},$ $\mathbb{C}$ ou une extension finie de $\mathbb{Q}_p.$
Supposons que $X$ est lisse.

Alors pour tout 0-cycle $z'$ sur $X$, il existe un 0-cycle $z$ sur $X$ tel que $\pi_{*}(z)$ est séparable et tel que $z$ est suffisamment proche de $z'$.
\end{lem}

\begin{lem}[{\cite[Prop. 3.1.1]{Liang4}}]\label{comparisonBrauer}
Soit $V$ une variété propre lisse et géométriquement rationnellement connexe définie sur $k.$
Alors il existe une extension finie $l$ de $k$ telle que,
pour toute $K\in\mathcal{K}_{l}$, l'homomorphisme induit par restriction 
$$\Br (V)/\Br (k)\to \Br (V_K)/\Br (K)$$
est un isomorphisme de groupes finis.
\end{lem}

\begin{lem}\label{surjBr}
Soient $X,Y$ des variétés géométriquement rationnellement connexes.  Alors il existe une extension finie $l$ de $k$ telle que, pour toute $K\in\mathcal{K}_{l}$, la composition des  homomorphismes suivants est surjective
$$\frac{\Br (X\times Y)}{\Br (k)}\to \frac{\Br (X_{K}\times Y_{K})}{\Br (K)}\buildrel{\theta^{*}}\over\To \frac{\Br (Y_{K})}{\Br (K)},$$
où le deuxième homomorphisme est la restriction à la fibre de $\pi_{1}:X\times Y\to X$ en un $K$-point rationnel arbitraire $\theta$ de $X_{K}$. 
\end{lem}

\begin{rem}
Le lemme ne dit rien sur l'existence de $K$-points rationnels sur $X_{K}$.
\end{rem}

\begin{proof}
Une fois qu'il existe un $K$-point rationnel $\theta$ de $X_{K}$, la projection $\pi_{2K}:X_{K}\times Y_{K}\to Y_{K}$ admet une section induite par $\theta$. D'où le deuxième homomorphisme est surjectif. On choisit $l$ comme dans Lemme \ref{comparisonBrauer} avec $V=X\times Y$, le premier  homomorphisme est un isomorphisme.
\end{proof}

\begin{lem}[Théorème d'irréductibilité de Hilbert pour les 0-cycles]\label{Hilbert_irred}
Soit $S$ un ensemble fini non-vide de places de $k$. Soit $L$ une extension finie de $k$.
Pour toute $v\in S$, soit $z_{v}$ un 0-cycle séparable  effectif de degré $d>0$ de support contenu dans $\mathbb{A}^{1}=\P^{1}\setminus\infty$.

Alors il existe un point fermé $\lambda$ de $\mathbb{A}^{1}$ tel que
\begin{itemize}
\item[-] $k(\lambda)\in\K_{L}$;
\item[-] vu comme un 0-cycle, $\lambda$ est suffisamment proche de  $z_{v}$ pour toute $v\in S$.
\end{itemize}
\end{lem}

\begin{proof}
C'est un cas particulier du \cite[Lem. 3.4]{Liang1} en prenant le sous-ensemble hilbertien généralisé de $\P^{1}$ defini par $\P^{1}_{L}\to\P^{1}$.
\end{proof}

\begin{proof}[Démonstration du Théorème \ref{mainthm}]
Quitte à augmenter $L$, on peut supposer que l'extension finie $L$ de $k$ contient 
\begin{itemize}
\item[-] une extension $l$ obtenue dans Lemme \ref{surjBr}, 
\item[-] et une extension $l$ obtenue dans Lemme \ref{comparisonBrauer} appliqué avec $V=X$.
\end{itemize}

Dans la preuve, on utilisera les notations suivantes $Z=X\times Y$, $Z'=Z\times\P^{1}$, $X'=X\times\P^{1}$, et $\Pi=\pi_{1}\times id_{\P^{1}}:Z'\to X'$. Afin de démontrer l'exactitude de $(\E)$ pour $Z$, il suffit de la démontrer pour $Z'$ car leurs groupes de Brauer et leurs groupes de Chow des 0-cycles sont naturellement isomorphes respectivement. Pour les variétés géométriquement rationnellement connexes, un argument de Wittenberg montre que l'exactitude de $(\E)$ pour $Z'$, fibrée comme $Z'\to\P^{1}$, est une conséquence de l'énoncé suivant concernant tout ensemble fini de places $S\subset\Omega_{k}$, voir \cite[Prop. 3.1]{Wittenberg} et voir aussi le début de la preuve du \cite[Thm. 2.5]{Liang-fibRC}.
\begin{itemize}
\item[$(P'_{S})$]  Soit $\{z_{v}\}_{v\in\Omega_{k}}$ une famille de 0-cycles locaux de degré $\delta$ sur $Z'$. Si elle est orthogonale à $\Br (Z')$, alors pour tout entier $n>0$, il existe un 0-cycle global $z$ de  $Z'$ de degré $\delta$ tel que, pour toute $v\in S$, on a $z=z_{v}$ dans $\CH_{0}(Z'_{v})/n$.
\end{itemize}

Montrons cette propriété pour tout ensemble fini $S\subset\Omega_{k}$.

Comme $X$ est géométriquement rationnellement connexe, le quotient $\Br (X)/\Br (k)$ est fini. On fixe un ensemble fini de représentants $\Gamma_{X}\subset\Br (X)$, on identifie $\Br (X)$ à $\Br (X')$ et on voit $\Gamma_{X}$ comme un sous-ensemble de $\Br (X')$. On fixe aussi un ensemble fini de représentants $\Gamma_{Z}\subset\Br (Z)=\Br (Z')$ de $\Br (Z)/\Br (k)$. Grâce à l'argument de bonne réduction, on peut aussi supposer que $S$ est suffisamment grand tel que l'évaluation locale de tout élément de $\Gamma_{X}$ et de $\Gamma_{Z}$ en tout 0-cycle vaut $0$ pour $v\notin S$. Comme $Y$ et $Z$ sont géométriquement intègres et lisses, quitte à augmenter $S$ on peut aussi supposer que $Y(K)\neq\emptyset$ et $Z(K)\neq\emptyset$ pour toute extension finie $K$ de $k_{v}$ et pour toute $v\notin S$ d'après l'estimation de Lang\textendash Weil et le lemme de Hensel. 

Soit $\{z_{v}\}_{v\in\Omega_{k}}$ une famille de 0-cycles locaux sur $Z'$, de degré $\delta$, et orthogonale à $\Br (Z')$.
Pour tout  $b_{X}\in\Gamma_{X}\subset\Br (X')$, on obtient 
$$\sum_{v\in\Omega_{k}}\inv_{v}(\langle z_{v},\Pi^{*}(b_{X})\rangle_{v})=0,$$
alors
$$\sum_{v\in S}\inv_{v}(\langle z_{v},\Pi^{*}(b_{X})\rangle_{v})=0.$$
De même, $$\sum_{v\in S}\inv_{v}(\langle z_{v},b_{Z}\rangle_{v})=0$$
pour tout $b_{Z}\in\Gamma_{Z}$.

Soit $m$ un entier positif qui annihile tout élément de $\Gamma_{X}$ et de $\Gamma_{Z}$. Fixons un point fermé $P$ de $Z'$, notons par $\delta_{P}=[k(P):k]$ le degré de $P$.

Pour toute $v\in S$, on écrit $z_{v}=z_{v}^{+}-z_{v}^{-}$ où $z_{v}^{+}$ et $z_{v}^{-}$ sont 0-cycles effectifs de supports disjoints.
On pose $z_{v}^{1}=z_{v}+mn\delta_{P}z_{v}^{-}=z_{v}^{+}+(mn\delta_{P}-1)z_{v}^{-}$, alors $\deg(z_{v}^{1})\equiv\delta\mod{mn\delta_{P}}$ et ils sont tous effectifs. On ajoute à chaque $z_{v}^{1}$ un multiple convenable de 0-cycle $mnP_{v}$ où $P_{v}=P\times_{\Spec(k)}\Spec(k_{v})$ et on obtient $z_{v}^{2}$ du même degré $\Delta$, avec $\Delta\equiv\delta\mod{mn\delta_{P}}$ pour toute $v\in S$. Vu la choix de $m$, on trouve $\langle z_{v}^{2}, \Pi^{*}(b_{X})\rangle_{v}=\langle z_{v}, \Pi^{*}(b_{X})\rangle_{v}$ pour toute $v\in S$ et tout $b\in\Gamma_X$. On applique le Lemme \ref{moving lemma} pour $pr_{\P^{1}}\circ\Pi: Z'\to\P^{1}$ et on obtient un 0-cycle effectif $z_{v}^{3}$ suffisamment proche de $z_{v}^{2}$ tel que $(pr_{\P^{1}}\circ \Pi)_{*}(z_{v}^{3})$ est séparable. A fortiori, $\Pi_{*}(z_{v}^{3})$ sont des 0-cycles séparables effectifs sur $X'$. D'après la continuité de l'évaluation locale, on trouve $\langle z_{v}^{3},\Pi^{*}(b)\rangle_{v}=\langle z_{v},\Pi^{*}(b)\rangle_{v}$. On vérifie que $z_{v}$, $z_{v}^{1}$, $z_{v}^{2}$, et $z_{v}^{3}$ ont la même image dans $\CH_{0}(Z'_{v})/n$.

On choisit un point rationnel $\infty\in\P^{1}(k)$ en dehors des supports de $(pr_{\P^{1}}\circ \Pi)_{*}(z^{3}_{v})$ pour toute $v\in S$. 
D'après le théorème d'irréductibilité de Hilbert pour les 0-cycles (Lemme \ref{Hilbert_irred}) appliqué à $pr_{\P^{1}}\circ \Pi:Z'\to\P^{1}$, on trouve un point fermé $\lambda\in\mathbb{A}^{1}$ tel que $k(\lambda)\in\K_{L}$ et $\lambda$ est suffisamment proche de $(pr_{\P^{1}}\circ \Pi)_{*}(z_{v}^{3})$. Plus précisément, on écrit
$\lambda_v=\lambda\times_{\mathbb{P}^1}{\mathbb{P}_{k_{v}}^1}=\bigsqcup_{w\mid v,w\in\Omega_{k(\lambda)}}\Spec(k(\lambda)_w)$
pour $v\in\Omega_k,$ l'image de $\lambda$ dans $Z_0(\mathbb{P}^1_{k_{v}})$ s'écrit comme
$\lambda_v=\sum_{w\mid v,w\in\Omega_{k(\lambda)}}P_w$ où $P_w=\Spec(k(\lambda)_w)$ est un point fermé de ${\mathbb{P}^1_v}$ de corps résiduel $k(\lambda)_w.$
Pour toute $v\in S,$ le 0-cycle $\lambda_v$ est suffisamment proche de $(pr_{\P^{1}}\circ \Pi)_*(z^3_v),$ où
le 0-cycle séparable effectif $(pr_{\P^{1}}\circ \Pi)_*(z^3_v)$ s'écrit comme $\sum_{w\mid v,w\in\Omega_{k(\lambda)}}Q_w$ avec $Q_w$ distincts deux à deux.
Alors $k(\lambda)_w=k_v(P_w)=k_v(Q_w),$  et $P_w$ est suffisamment proche de $Q_w\in \mathbb{P}^1_v(k(\lambda)_w).$
On sait que
$z^3_v$ s'écrit comme $\sum_{w\mid v,w\in\Omega_{k(\lambda)}}M^0_w$ with $k_v(M^0_w)=k(\lambda)_w$ et
$M^0_w\in Z'_v(k(\lambda)_w)$ se trouve sur la fibre de $pr_{\P^{1}}\circ \Pi$ en point fermé $Q_w.$
Le théorème des fonctions implicites implique qu'il existe un $k(\lambda)_w$-point lisse  $M_w$ sur la fibre $(pr_{\P^{1}}\circ \Pi)^{-1}(\lambda)$ suffisamment proche de $M^0_w$ pour toute $w\in S\otimes_kk(\lambda).$ Alors les points fermés $M_w$ et $M_w^0$ ont la même image dans $CH_0(X_v)/n$.

D'après la continuité de l'évaluation locale, pour tout $b_{X}\in \Gamma_{X}$ on obtient l'égalité sur $X'$ 
$$\sum_{w\in S\otimes_kk(\lambda)}\inv_w( \langle \Pi(M_w),b_{X} \rangle_{k(\lambda)_w})=0.$$
De même, sur $Z'$, pour tout $b_{Z}\in\Gamma_{Z}$
$$\sum_{w\in S\otimes_kk(\lambda)}\inv_w( \langle M_w,b_{Z} \rangle_{k(\lambda)_w})=0.$$

Sur la $k(\lambda)$-variété $(pr_{\P^{1}}\circ\Pi)^{-1}(\lambda)=Z\times\lambda\simeq Z_{k(\lambda)}$, on fixe un $k(\lambda)_{w}$- point rationnel $M_{w}$ pour toute $w\in\Omega_{k(\lambda)}\setminus S\otimes_{k}k(\lambda)$.  On pose $N_{w}=\Pi(M_{w})$ pour toute $w\in \Omega_{k(\lambda)}$. D'après la choix de $S$, on obtient alors l'égalité pour tout $b_{X}\in\Gamma_{X}$
$$\sum_{w\in \Omega_{k(\lambda)}}\inv_w( \langle N_w,b_{X} \rangle_{k(\lambda)_w})=0.$$

Par l'abus de notations, on note par $b_{X}$ aussi son image de la restriction $\Br (X')\to\Br (X\times\lambda)$. D'après la fonctorialité, cette égalité peut être vue comme l'accouplement de Brauer\textendash Manin sur $X\times \lambda$. On rappel que $k(\lambda)\in\K_{L}$, la restriction $\Br (X)/\Br (k)\to\Br (X\times\lambda)/\Br (k(\lambda))$ est un isomorphisme. Donc $\{N_{w}\}_{w\in\Omega_{k(\lambda)}}$ est orthogonal au groupe de Brauer de $X\times\lambda$.

D'après l'hypothèse, l'approximation faible donne un $k(\lambda)$-point rationnel $\theta$ sur $X\times\lambda$ suffisamment proche de $N_{w}$ pour toute $w\in S\otimes_{k}k(\lambda)$,
on peut aussi demander que la fibre $\Pi^{-1}(\theta)$ soit lisse.  
D'après le théorème des fonctions implicites $\Pi^{-1}(\theta)$ admet des $k(\lambda)$-points rationnels $M_{w}^{\theta}$ suffisamment proche de $M_{w}$ pour toute $w\in S\otimes_{k}k(\lambda)$. D'après la continuité de l'accouplement de Brauer\textendash Manin, on obtient
$$\sum_{w\in S\otimes_kk(\theta)}\inv_w( \langle M^{\theta}_w,b_{Z} \rangle_{k(\theta)_w})=0$$
pour tout $b_{Z}\in\Gamma_{Z}$. Par la choix de $S$, pour toute $w\in\Omega_{k(\lambda)}\setminus S\otimes_{k}k(\lambda)$ il existe des $k(\lambda)_{w}$-points rationnels $M^{\theta}_{w}$ sur la fibre $\Pi^{-1}(\theta)\simeq Y_{k(\lambda)}$, et de plus 
$$\sum_{w\in \Omega_{k(\theta)}}\inv_w( \langle M^{\theta}_w,b_{Z} \rangle_{k(\theta)_w})=0.$$
Ceci peut être vu comme une égalité sur la $k(\theta)$-variété $\Pi^{-1}(\theta)$.
Le corps résiduel $k(\theta)$ est linéairement disjoint avec $L$ sur $k$, donc  $\Gamma_{Z}$ se surjecte sur $\Br(\Pi^{-1}(\theta))/\Br (k(\theta))$. Par l'hypothèse que $(\E)$ est exacte pour $\Pi^{-1}(\theta)\simeq Y_{k(\theta)}$, il existe alors un 0-cycle global $z'$ de degré $1$ sur la $k(\theta)$-variété $\Pi^{-1}(\theta)$ ayant la même image que $M_{w}^{\theta}$ dans $\CH_{0}(\Pi^{-1}(\theta)_{{w}})/n$  pour toute $w\in S\otimes_{k}k(\theta)$, \textit{cf.} \cite[Thm. A]{Liang4}. Vu comme un 0-cycle de $Z'$, le 0-cycle $z'$ est de degré $\Delta\equiv\delta\mod mn\delta_{P}$. Quitte à soustraire un multiple convenable du point fermé  $P$, on obtient un 0-cycle $z$ de degré $\delta$. On vérifie que $z$ et $z_{v}$ ont la même image dans $\CH_{0}(Z'_{v})/n$ pour toute $v\in S$. Ceci termine la démonstration.
\end{proof}

\small


\bibliographystyle{alpha}
\bibliography{mybib1fr}
\end{document}